\documentclass[letterpaper,12pt,normalheadings]{scrartcl}

\usepackage[american]{babel}
\usepackage[utf8]{inputenc}
\usepackage[T1]{fontenc}
\usepackage[shortcuts]{extdash}
\usepackage{geometry}
\usepackage{hyperref}
\hypersetup{colorlinks=true, linkcolor=chroma4blue, citecolor=chroma4blue, urlcolor=black, filecolor=black}
\usepackage{lmodern}
\usepackage{parskip}
\usepackage{caption}
\usepackage{subcaption}
\usepackage{authblk}
\usepackage{graphicx}
\usepackage{amsmath}
\usepackage{amssymb}
\usepackage{amsthm}
\usepackage{dsfont}
\usepackage{mathtools}
\usepackage{booktabs}
\usepackage{tabulary}
\usepackage{xspace}
\usepackage{tikz}
\usetikzlibrary{patterns}
\usetikzlibrary{arrows.meta,bending}
\usetikzlibrary{positioning}
\usetikzlibrary{fadings}
\usepackage{pgfplots}
\pgfplotsset{compat=1.17}
\usepackage{nicematrix}
\usepackage{algorithm}
\usepackage[noend]{algpseudocode}
\usepackage{csquotes}

\algrenewcommand\algorithmicrequire{\textbf{Input:}}
\algrenewcommand\algorithmicensure{\textbf{Output:}}

\tikzstyle{vertex} = [black, fill, circle, text width=2mm, inner sep=0pt]

\definecolor{chroma4blue}{RGB}{63,97,135}
\definecolor{chroma4green}{RGB}{65,143,126}
\definecolor{chroma4sand}{RGB}{171,166,125}
\definecolor{chroma4gray}{RGB}{199,199,199}

\NiceMatrixOptions{transparent, xdots/shorten=0.6em}

\mathcode`@="8000
{\catcode`\@=\active\gdef@{\mkern1mu}}

\DeclareMathOperator{\diag}{diag}

\mathcode`@="8000
{\catcode`\@=\active\gdef@{\mkern1mu}}

\makeatletter
\def\moverlay{\mathpalette\mov@rlay}
\def\mov@rlay#1#2{\leavevmode\vtop{%
   \baselineskip\z@skip \lineskiplimit-\maxdimen
   \ialign{\hfil$\m@th#1##$\hfil\cr#2\crcr}}}
\newcommand{\charfusion}[3][\mathord]{
    #1{\ifx#1\mathop\vphantom{#2}\fi
        \mathpalette\mov@rlay{#2\cr#3}
      }
    \ifx#1\mathop\expandafter\displaylimits\fi}
\makeatother

\theoremstyle{definition}
\newtheorem{lemma}{Lemma}
\newtheorem*{example*}{Example}

\newtheorem{corollary}[lemma]{Corollary}
\newtheorem{theorem}[lemma]{Theorem}

\newtheorem{definition}[lemma]{Definition}

\numberwithin{lemma}{section}
\numberwithin{fact}{section}
\numberwithin{equation}{section}

\makeatletter
\newcommand{\pushright}[1]{\ifmeasuring@#1\else\omit\hfill$\displaystyle#1$\fi\ignorespaces}
\newcommand{\pushleft}[1]{\ifmeasuring@#1\else\omit$\displaystyle#1$\hfill\fi\ignorespaces}
\makeatother

\makeatletter
\newlength{\negph@wd}
\DeclareRobustCommand{\negphantom}[1]{%
  \ifmmode
    \mathpalette\negph@math{#1}%
  \else
    \negph@do{#1}%
  \fi
}
\newcommand{\negph@math}[2]{\negph@do{$\m@th#1#2$}}
\newcommand{\negph@do}[1]{%
  \settowidth{\negph@wd}{#1}%
  \hspace*{-\negph@wd}%
}
\makeatother

\addtokomafont{sectioning}{\rmfamily}

\title{\Large \textbf{Fast ultrametric matrix-vector multiplication}}
\author{\normalsize Tobias Hofmann\textsuperscript{1}, Andy Oertel\textsuperscript{2}}
\date{}
\affil{\footnotesize
Chemnitz University of Technology\\ \textsuperscript{1}\texttt{tobias.hofmann@math.tu-chemnitz.de},\\ Lund University\\ \textsuperscript{2}\texttt{andy.oertel@cs.lth.se}
}

\begin{document}
\maketitle

\begin{abstract}
\noindent
\textbf{Abstract.} We study the properties of ultrametric matrices aiming to design methods for fast ultrametric matrix-vector multiplication. We show how to encode such a matrix as a tree structure in quadratic time and demonstrate how to use the resulting representation to perform matrix-vector multiplications in linear time. Accompanying this article, we provide an implementation of the proposed algorithms and present empirical results on their practical performance. \\


\end{abstract}

\section{Introduction} Ultrametricity is a remarkable, occasionally a little counterintuitive, but often natural and interesting property. Examples in which ultrametric distances arise range from the $p$-adic number system to phylogenetic trees, which is illustrated nicely by Holly~\cite{holly2001pictures}. Accordingly, ultrametric matrices appear in various mathematical fields. The monograph of Dellacherie, Mart{\'\i}nez, and Mart{\'\i}n~\cite{dellacherie2014inverse} describes how ultrametric matrices are related to M-matrices and underlines their relevance in discrete potential theory or the analysis of Markov chains. Another remarkable property, established by Mart{\'\i}nez, Michon, and San Mart{\'\i}n in~\cite{martinez1994inverse}, is that ultrametric matrices are nonsingular and their inverses are strictly diagonally dominant Stieltjes matrices. We learned about their rich properties while investigating edge-connectivity matrices, whose off-diagonal entries satisfy an ultrametric inequality. This is a classical result of Gomory and Hu~\cite{gomoryhu1961multi}, which links ultrametricity with topics from combinatorics and spectral graph theory, as is discussed in Hofmann and Schwerdtfeger~\cite{hofmann2021edge}. Furthermore, ultrametric matrices play a role in statistics and data analysis. Chehreghani~\cite{chehreghani2020unsupervised} develops a machine learning framework that builds on minimax, and herewith ultrametric, distance measures. Lauritzen, Uhler, and Zwiernik~\cite{lauritzen2019maximum} show that ultrametric matrices are relevant in maximum likelihood estimation problems for specific Gaußian distributions. Another example is an ultrametric spectral clustering approach developed by Little, Maggioni, and Murphy~\cite{little2020path}.

As interest in applications involving ultrametric matrices grows, the question of how to perform efficient ultrametric matrix computations arises. This is the focus of this article. Building on the well-known fact that ultrametric matrices are completely reducible, our main contributions are explicit algorithmic ideas how to encode an ultrametric matrix as its associated tree structure and how to use this data structure to perform fast ultrametric matrix-vector multiplications.

{\bf Outline.} We review basic facts about ultrametric matrices and point out how these matrices are related to tree structures in Section~\ref{sec:properties}. Section~\ref{sec:multiplication} is about utilizing these data structures to perform fast matrix-vector multiplications. In Section~\ref{sec:empirical}, we summarize results about the performance of the methods we propose. Accompanying our computational insights, we provide an implementation of our algorithms. \\

We conclude this section with certain concepts and notations that are particularly important for our investigation. We use~$\mathds{1}$ to denote the all ones column vector of appropriate dimensions. The symbol~$e_i$ represents the standard column basis vector of appropriate dimensions, whose entries are defined via~$(e_i)_j\coloneqq 1$ if~$i=j$ and~$(e_i)_j\coloneqq 0$ if~$i\neq j$. For a matrix~$A\in\mathbb{R}^{n\times n}$, we use index sets~$I,J\subset \{1,\ldots,n\}$ to specify~$A_{I J}$ as the submatrix that contains those rows of~$A$ that belong to the indices in~$I$ and those columns of~$A$ that belong to indices in~$J$. If~$I=J$, we may use the shorthand~$A_I$ instead of~$A_{I I}=A_{I J}$. We denote diagonal matrices whose entries are given by a sequence~$(a_i)_{i=1}^n$ by~$\diag(a_i : i=1,\ldots,n)$. For graph theoretical terminology, we refer to the monograph of Diestel~\cite{diestel2017graph}.

\section{Basic Properties of ultrametric matrices}\label{sec:properties}

The investigation of ultrametric matrices gained in importance with the article by Mart{\'\i}nez, Michon, and San Mart{\'\i}n~\cite{martinez1994inverse} who give in essence the following definition.\newpage
\begin{definition} \label{def:ultrametric}
A nonnegative symmetric matrix~$A = [a_{ij}]\in\mathbb{R}^{n\times n}$ is said to be \emph{ultrametric} if it satisfies the inequalities
\begin{itemize}
    \item[(a)] $a_{ij}\ge\min\{a_{ik},a_{kj}\}$\negphantom{$a_{ij}\ge\min\{a_{ik},a_{kj}\}$}\phantom{$a_{ii}\ge\max\{a_{ij}:j\in\{1,\ldots,n\}\setminus\{i\}\}$}\quad for all~$i\neq j\neq k\neq i$,
    \item[(b)] $a_{ii}\ge\max\{a_{ij}:j\in\{1,\ldots,n\}\setminus\{i\}\}$\quad for all~$i$.
\end{itemize}
\end{definition}
The inequalities in~(a) are known as \emph{ultrametric inequalities} and a matrix that satisfies~(b) is referred to as \emph{column pointwise diagonal dominant}. If~$A$ satisfies~(a), but not necessarily~(b), we call~$A$ \emph{essentially} ultrametric. If~$A$ satisfies the inequalities in~(b) with equality, we call~$A$ \emph{special} ultrametric, and if~$A$ satisfies the inequalities in~(b) strictly, we call~$A$ \emph{strictly} ultrametric. A matrix of size~$n=1$ is strictly ultrametric only if its entry is positive, whereas there is no such convention for special or essentially ultrametric matrices.

The focus in the article of Mart{\'\i}nez, Michon, and San Mart{\'\i}n~\cite{martinez1994inverse} is on strictly ultrametric matrices, whereas Fiedler~\cite{fiedler2000special} studied special ultrametric matrices, which can be seen as extremal matrices in the boundary of the set of ultrametric matrices. The term essentially ultrametric is to emphasize situations in which specific diagonal entries are not of interest. For example, this is the case for the edge-connectivity matrices in~\cite{hofmann2021edge}. A central property of strictly ultrametric matrices is that they are nonsingular and their inverses are diagonally dominant M-matrices. Mart{\'\i}nez, Michon, and San Mart{\'\i}n prove this fact in~\cite{martinez1994inverse} by probabilistic arguments. A linear algebra proof is given by Nabben and Varga~\cite{nabben1994linear}. Their arguments essentially rely on the fact that ultrametric matrices are \emph{completely reducible}, which is what they state in the following way.
\begin{theorem}\label{thm:nabben} Let~$A=[a_{ij}]$ be a nonnegative symmetric matrix in~$\mathbb{R}^{n\times n}$. If~$n>1$, then~$A$ is essentially ultrametric if and only if there is an integer~$k$ with~$1\le k < n$ and a suitable permutation matrix~$P\in\mathbb{R}^{n\times n}$ such that
\begin{equation*}
    P \big(A-\min\{a_{ij}:i\neq j\}@@\mathds{1}\mathds{1}^\top\big) P^\top = \begin{bmatrix}
    B & 0 \\
    0 & C
    \end{bmatrix},
\end{equation*}
where~$B$ and~$C$ are essentially ultrametric matrices in~$\mathbb{R}^{k\times k}$ and~$\mathbb{R}^{(n-k)\times(n-k)}$, respectively.
\end{theorem}
Note that in~\cite{nabben1994linear} the above statement is formulated for a strictly ultrametric matrix~$A$. In this case, the matrices~$B$ and~$C$ follow to be strictly ultrametric as well. However, the idea of the proof presented in~\cite{nabben1994linear} actually does not require any particular diagonal entries. Fiedler~\cite{fiedler2000special}, for example, follows the same line of reasoning to obtain Theorem~\ref{thm:nabben} except that~$A$, $B$, and~$C$ are special ultrametric. In our statement above, we simply ignore the diagonal entries of~$A$ and accordingly claim nothing about the diagonal entries of~$B$ and~$C$. Also note that whereas Theorem~\ref{thm:nabben} only states the existence of a suitable integer~$k$ and a permutation matrix~$P$, the focus of this article is on algorithms to determine~$P$ explicitly. The following simple but useful observation is our first step in that direction.
\begin{lemma}\label{lem:min_on_row}
In each row and column of an essentially ultrametric matrix~$A=[a_{ij}]$ there is an entry equal to~$\min\{a_{ij}:i\neq j\}$.
\end{lemma}
\begin{proof}
Theorem~\ref{thm:nabben} tells us that there is an entry equal to zero in each row and column of~$P \big(A-\min\{a_{ij}:i\neq j\}@@\mathds{1}\mathds{1}^\top\big) P^\top$, where~$P$ is some permutation matrix. Permuting rows and columns, however, preserves this property. So there is an entry equal to zero in each row and column of~$A-\min\{a_{ij}:i\neq j\}@@\mathds{1}\mathds{1}^\top$. In other words, there is an entry equal to~$\min\{a_{ij}:i\neq j\}$ in each row and column of~$A$.
\end{proof}
Theorem~\ref{thm:nabben} essentially is a decomposition statement showing that there is a tree structure inherent in an ultrametric matrix. Lemma~\ref{lem:min_on_row} emphasizes the fact that we can find the global minimum of the off-diagonal entries of an ultrametric matrix in each of its rows or columns. This is the reason why we may process such a matrix row by row when asking for its underlying tree structure.
\hypertarget{tree_link}{\begin{algorithm}[ht]
\caption{Ultrametric Tree Construction}\label{algo:tree_gen}
\begin{algorithmic}[1]
\Require{essentially ultrametric matrix~$A=[a_{ij}]\in\mathbb{R}^{n\times n}$}
\Ensure{ultrametric tree~$(V,E)$ associated with~$A$}
\Statex
\State $V\gets \{r\}$\label{line:initV}
\State $E\negphantom{E}\phantom{V}\gets \emptyset$\label{line:initE}
\State $I(r) \gets \{1,\ldots,n\}$ \label{line:initI}
\State \Call{TreeRecursion}{$r$}
\Statex
\Procedure{TreeRecursion}{$u$}
\State $i \gets \min (I(u))$
\If{$|I(u)|=1$}\label{line:termination}
\State $f(u) \gets a_{ii}$ \label{line:leaf}
\Else
\State $f(u) \gets \min \{ a_{ij}:j \in I(u)\setminus\{i\}\}$ \label{line:inner}
\State $V\gets V\cup\{v,w\}$
\State $E\negphantom{E}\phantom{V}\gets E\negphantom{E}\phantom{V}\cup\{(u, v), (u, w)\}$
\State $I(v)\phantom{f(u)}\negphantom{I(v)} \gets \{ j \in I(u) : a_{ij} > f(u) \}\cup\{i\}$ \label{line:bigger}
\State $I(w)\phantom{f(u)}\negphantom{I(w)} \gets \{ j \in I(u)\setminus\{i\} : a_{ij} = f(u) \}$
\State \Call{TreeRecursion}{$v$}
\State \Call{TreeRecursion}{$w$}
\EndIf
\EndProcedure
\end{algorithmic}
\end{algorithm}}
For the explicit computation of a tree~$(V,E)$ associated with an ultrametric matrix, we propose Algorithm~\ref{algo:tree_gen}. Here, an edge~$ij\in E$ is to be understood as directed and we address~$i$ as \emph{parent} and~$j$ is its~\emph{child}. Furthermore, each vertex~$u \in V$ takes an index set~$I(u)$ and a value~$f(u)$.
For an example of how Algorithm 1 works, we may take a look at Figure~\ref{fig:tree_example}. It shows an essentially ultrametric matrix~$A$ and the tree that results when applying Algorithm~\ref{algo:tree_gen} to it. It is indeed possible to go on pruning the resulting tree while retaining all the information about the matrix $A$ by contracting a vertex~$v$ and its parent~$u$ if~$f(u)=f(v)$. This may be useful in some situations and is an option our implementation supports. In general, however, pruning may not be possible at all and as it would otherwise overcomplicate our notation, we consider unpruned trees when analyzing the characteristics of Algorithm~\ref{algo:tree_gen}.
\begin{figure}
\begin{center}
\begin{tikzpicture}

\node at (-3.5,1.5) {$A = \begin{bmatrix} 
0 & 1 & 3 & 1 \\
1 & 3 & 1 & 2 \\
3 & 1 & 5 & 1 \\
1 & 2 & 1 & 1 
\end{bmatrix}$};

\footnotesize

\node[thick, draw=chroma4green, inner sep=1.5mm, circle] (1) at (4,3.3) {};
\node at (1) {$r$};
\node[thick, draw=chroma4green, inner sep=1.5mm, circle] (2) at (1.25,1.5) {};
\node at (2) {$u$};
\node[thick, draw=chroma4green, inner sep=1.5mm, circle] (3) at (6.75,1.5) {};
\node at (3) {$a$};
\node[thick, draw=chroma4green, inner sep=1.5mm, circle] (4) at (0,0) {};
\node at (4) {$v$};
\node[thick, draw=chroma4green, inner sep=1.5mm, circle] (5) at (2.5,0) {};
\node at (5) {$w$};
\node[thick, draw=chroma4green, inner sep=1.5mm, circle] (6) at (5.5,0) {};
\node at (6) {$b$};
\node[thick, draw=chroma4green, inner sep=1.5mm, circle] (7) at (8,0) {};
\node at (7) {$c$};

\node[above=-1.25mm] at (1) {$I(r)\!=\!\{1,2,3,4\}\hspace{7mm}f(r)\!=\!1\negphantom{f(r)\!=\!1}\phantom{I(r)\!=\!\{1,2,3,4\}}$~};

\node[above=-1.25mm] at (2) {$I(u)\!=\!\{1,3\}\hspace{7mm}f(u)\!=\!3\negphantom{f(u)\!=\!3}\phantom{I(u)\!=\!\{1,3\}}$};

\node[above=-1.25mm] at (3) {$I(a)\!=\!\{2,4\}\hspace{7mm}f(a)\!=\!2\negphantom{f(a)\!=\!2}\phantom{I(a)\!=\!\{2,4\}}$};

\node[below=1.5mm] at (4) {$\begin{aligned}
     I(v)\!&=\!\{1\}\\[-1mm]
     f(v)\!&=\!0
  \end{aligned}$};

\node[below=1.5mm] at (5) {$\begin{aligned}
     I(w)\!&=\!\{3\}\\[-1mm]
     f(w)\!&=\!5
  \end{aligned}$};

\node[below=1.5mm] at (6) {$\begin{aligned}
     I(b)\!&=\!\{2\}\\[-1mm]
     f(b)\!&=\!3
  \end{aligned}$};

\node[below=1.5mm] at (7) {$\begin{aligned}
     I(c)\!&=\!\{4\}\\[-1mm]
     f(c)\!&=\!1
  \end{aligned}$};

\draw[->, >=latex, thick, black!75!white, rounded corners=1.5mm] (1) to (1.25,2.75) to (2);
\draw[->, >=latex, thick, black!75!white, rounded corners=1.5mm] (1) to (6.75,2.75) to (3);
\draw[->, >=latex, thick, black!75!white, rounded corners=1.5mm] (2) to (0,1.25) to (4);
\draw[->, >=latex, thick, black!75!white, rounded corners=1.5mm] (2) to (2.5,1.25) to (5);
\draw[->, >=latex, thick, black!75!white, rounded corners=1.5mm] (3) to (5.5,1.25) to (6);
\draw[->, >=latex, thick, black!75!white, rounded corners=1.5mm] (3) to (8,1.25) to (7);
\end{tikzpicture}
\end{center}
\caption{An essentially ultrametric matrix~$A$ and its associated tree constructed by Algorithm~\ref{algo:tree_gen}}\label{fig:tree_example}
\end{figure}
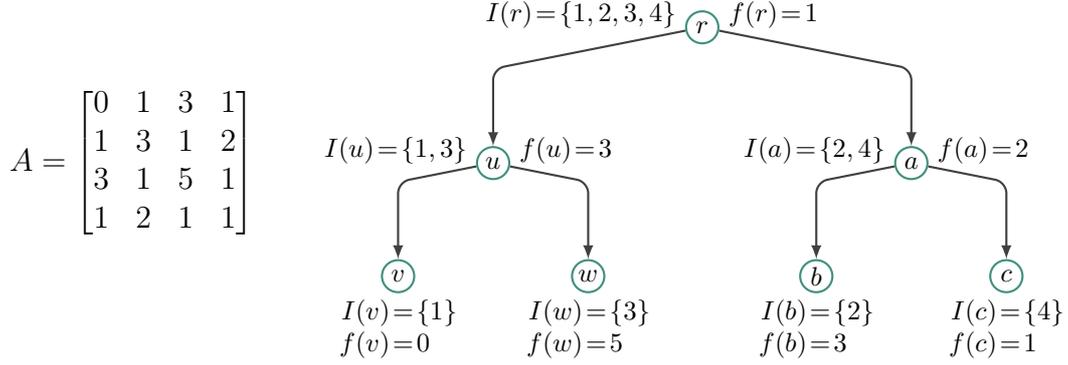

\begin{theorem}\label{thm:correct_tree_gen}
Algorithm~\ref{algo:tree_gen} that has been given an essentially ultrametric matrix~\mbox{$A=[a_{ij}]\in\mathbb{R}^{n\times n}$} as input terminates after $2n-1$ recursion calls and its output is a rooted directed tree~$(V,E)$ in which each vertex can be reached from the root~$r$ by a unique directed path. Moreover, the tree~$(V,E)$ has the following properties.
\begin{enumerate}
    \item A submatrix~$A_{I(u)}$ is essentially ultrametric for each~$u\in V$.\label{thm:correct_tree_gen1}
    \item For each~$i\in\{1,\ldots,n\}$, there is a leaf~$u\in V$ with~$I(u)=\{i\}$ and~$f(u)=a_{ii}$.\label{thm:correct_tree_gen2}
    \item If $u$ has a child $v$, then $f(u)=a_{ij}$ for all~$i\in I(v)$ and all~$j\in I(u)\setminus I(v)$.\label{thm:correct_tree_gen3}
\end{enumerate}
\end{theorem}
\begin{proof}
At first, we examine that for a vertex~$u$ with index set~$I(u)$ of size~$|I(u)|\geq 2$ a recursion step of Algorithm~\ref{algo:tree_gen} sets $i=\min (I(u))$,~$f(u)=\min\{a_{ij}:j\in I(u)\setminus\{i\}\}$, and divides the set~$I(u)$ into two subsets
\begin{align*}
    I(v)&=\{j\in I(u):a_{ij}>f(u)\}\cup\{i\}\quad\text{and}\\
    I(w)&=\{j\in I(u)\setminus\{i\}:a_{ij}=f(u)\}.
\end{align*}
So we conclude that~\mbox{$I(v)\neq\emptyset$}, \mbox{$I(w)\neq\emptyset$}, $I(v)\cap I(w)=\emptyset$, and~\mbox{$I(u)=I(v)\cup I(w)$}. This means that subsequent recursion steps operate on a nonempty, but smaller index set. This also implies that Algorithm~\ref{algo:tree_gen}, initializing~$I(u)=\{1,\ldots,n\}$ in Line~\ref{line:initI}, has to process~\mbox{$n-1$} recursion steps that run through their else case to decompose the initial index set completely and eventually, the recursion is called with input~$u$ for which~$I(u)=\{i\}$ for each~$i\in\{1,\ldots,n\}$ at some point. This leads into the recursion's if case and thus causes the respective recursion branch to terminate. In such a case, the algorithm assigns~$f(u)=a_{ii}$ by Line~\ref{line:leaf}, which proves Statement~\ref{thm:correct_tree_gen2}. Since this happens~$n$ times, we count a total of~$2n-1$ recursion steps.

The graph~$(V,E)$ that Algorithm~\ref{algo:tree_gen} constructs is initialized by~\mbox{$V=\{r\}$} and~\mbox{$E=\emptyset$} in Lines~\ref{line:initV} and~\ref{line:initE}. This graph gets assigned new vertices and edges only in the else case of our recursion and there we always append two vertices by two edges to the graph constructed up to that point. This provides us with a connected graph that contains~$2n-1$ vertices and~$2n-2$ edges. So Algorithm~\ref{algo:tree_gen} outputs a tree and since the direction in which the edges are included follows exactly the layout of the recursion tree, we find the vertex~$r$ that is initialized in line~\ref{line:initV} to be the root, from which all other vertices can be reached by a unique directed path.

To prove Statement~\ref{thm:correct_tree_gen1}, we proceed inductively. We are given that~$A_{I(r)}=A$ is essentially ultrametric. So let us consider a recursion step with input~$u$ for which we suppose that $|I(u)|\geq 2$ and $A_{I(u)}$ is essentially ultrametric. Denoting~$\ell\coloneqq|I(u)|$ as well as~$I(v)=\{i_1,\ldots,i_k\}$ and~$I(w)=\{i_{k+1},\ldots,i_\ell\}$, we define the permutation matrix
\begin{align*}
    P^\top&=[e_{i_1},\ldots,e_{i_k},e_{i_{k+1}},\ldots,e_{i_\ell}]
\intertext{an consider}
    P@@A_{I(u)}P^\top &= \begin{bmatrix}
        A_{I(v)} & A_{I(v)I(w)} \\
        A_{I(w)I(v)} & A_{I(w)} \\
    \end{bmatrix}.
\end{align*}
Since the algorithm sets~$i=\min(I(u))$ and~$I(w)=\{j\in I(u)\setminus\{i\}:a_{ij}=f(u)\}$, all the entries in row~$i$ of~$A_{I(v)I(w)}$ are equal to~$f(u)=\min\{a_{ij}:j\in I(u)\setminus\{i\}\}$. We already observed in Lemma~\ref{lem:min_on_row} that in this way we find the smallest global off-diagonal entry~$f(u)=\{a_{ij}:i\neq j\}$. Theorem~\ref{thm:nabben} thus tells us that indeed all the entries in~$A_{I(v)I(w)}$ are equal to~$f(u)$. Consequently, we observe that~$f(u)=a_{ij}$ for all~$i\in I(v)$ and all~$j\in I(u)\setminus I(v)$ and, by symmetry, that~$f(u)=a_{ij}$ for all~$i\in I(w)$ and all~$j\in I(u)\setminus I(w)$. This proves Statement~\ref{thm:correct_tree_gen3} since we have chosen $u$ to be an arbitrary vertex among those that have children. Furthermore, Theorem~\ref{thm:nabben} implies that $A_{I(v)}$ and $A_{I(w)}$ are again essentially ultrametric, which was to be shown for Statement~\ref{thm:correct_tree_gen1}.
\end{proof}

\begin{corollary}\label{cor:time_tree_gen}
Algorithm~\ref{algo:tree_gen} requires~$\mathcal{O}(n^2)$ floating-point operations to encode an essentially ultrametric matrix~$A\in\mathbb{R}^{n\times n}$ as its associated tree structure. 
\end{corollary}
\begin{proof}
The algorithm terminates after~$2n-1$ recursion calls by Theorem~\ref{thm:correct_tree_gen}. Each recursion step requires~$\mathcal{O}(n)$ floating-point operations to determine~$f(u)$,~$I(v)$, and~$I(w)$, as for each of them at most~$n$ comparisons have to be performed. So in total we count~$\mathcal{O}(n^2)$ floating-point operations.
\end{proof}

\section{Fast matrix-vector multiplication}\label{sec:multiplication}
The following algorithm is designed to perform fast matrix-vector multiplications for a matrix that is given in its ultrametric tree representation~$(V,E)$ constructed by Algorithm~\ref{algo:tree_gen}. As before, each vertex~$u\in V$ is provided with an index set~$I(u)$ and a value~$f(u)$. In addition, we assign values~$s(u)$,~$t(u)$ and~$p(u)$ in what follows.

\hypertarget{tree_mult_link}{\begin{algorithm}[ht]
\caption{Ultrametric Multiplication}\label{algo:tree_mul}
\begin{algorithmic}[1]
\Require{ultrametric tree~$(V,E)$ with root vertex~$r$ constructed by Algorithm~\ref{algo:tree_gen} for an essentially ultrametric matrix~$A\in\mathbb{R}^{n\times n}$, vector~$x\in\mathbb{R}^n$}
\Ensure{product~$y = Ax$}
\Statex
\State \Call{PartialProduct}{$r$,~$0$}
\State \Call{TotalProduct}{$r$,~$0$}
\Statex
\Procedure{PartialProduct}{$u$, $z$}
\If{$|I(u)| = 1$}
\State $s(u) \gets x_i$ where~$i\in I(u)$\label{line:assign_s}
\Else
\State $s(u) \gets \hspace{-9mm}\displaystyle\sum\limits_{~~~~~v \in V: (u, v) \in E}\hspace{-9mm} \Call{PartialProduct}{v, f(u)}$ \label{line:partial_product_recursion}
\EndIf
\State $t(u) \gets (f(u) - z)@@s(u)$\label{line:assign_t}
\State \Return{$s(u)$}
\EndProcedure
\Statex
\Procedure{TotalProduct}{$u$, $q$}
\State $p(u) \gets q + t(u)$\label{line:assign_p}
\If{$|I(u)| = 1$}
\State $y_i \gets p(u)$ where~$i\in I(u)$\label{line:assign_y}
\Else
\ForAll {$v \in V: (u, v) \in E$} \label{line:for}
\State $\Call{TotalProduct}{v, p(u)}$ \label{line:total_product_recursion}
\EndFor
\EndIf
\EndProcedure
\end{algorithmic}
\end{algorithm}}
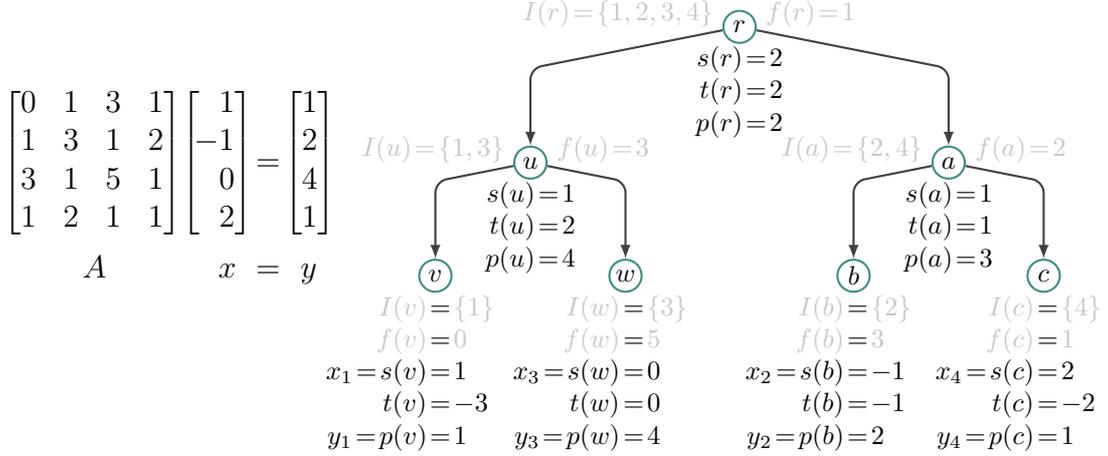
\begin{figure}
\begin{center}
\begin{tikzpicture}

\node at (-3.5,1.5) {$\begin{aligned}
        &\\
     \begin{bmatrix} 
0 & 1 & 3 & 1 \\
1 & 3 & 1 & 2 \\
3 & 1 & 5 & 1 \\
1 & 2 & 1 & 1 
\end{bmatrix} \begin{bmatrix} \!\negphantom{1}\phantom{-1}1 \\ \!-1 \\ \!\negphantom{0}\phantom{-1}0 \\ \!\negphantom{2}\phantom{-1}2 \end{bmatrix} &\!=\!\begin{bmatrix} 1 \\ 2 \\ 4 \\ 1 \end{bmatrix}\\
A\hspace{14.75mm}x\hspace{2.15mm} &\!=\!\hspace{2.15mm}y
  \end{aligned}$};

\footnotesize

\node[thick, draw=chroma4green, inner sep=1.5mm, circle] (1) at (4,3.3) {};
\node at (1) {$r$};
\node[thick, draw=chroma4green, inner sep=1.5mm, circle] (2) at (1.25,1.5) {};
\node at (2) {$u$};
\node[thick, draw=chroma4green, inner sep=1.5mm, circle] (3) at (6.75,1.5) {};
\node at (3) {$a$};
\node[thick, draw=chroma4green, inner sep=1.5mm, circle] (4) at (0,0) {};
\node at (4) {$v$};
\node[thick, draw=chroma4green, inner sep=1.5mm, circle] (5) at (2.5,0) {};
\node at (5) {$w$};
\node[thick, draw=chroma4green, inner sep=1.5mm, circle] (6) at (5.5,0) {};
\node at (6) {$b$};
\node[thick, draw=chroma4green, inner sep=1.5mm, circle] (7) at (8,0) {};
\node at (7) {$c$};

\node[above=-1.25mm] at (1) {\textcolor{chroma4gray}{$I(r)\!=\!\{1,2,3,4\}\hspace{7mm}f(r)\!=\!1\negphantom{f(r)\!=\!1}\phantom{I(r)\!=\!\{1,2,3,4\}}$}};

\node[below=1mm] at (1) {$\begin{aligned}
     s(r)\!&=\!2 \\[-1mm]
     t(r)\!&=\!2 \\[-1mm]
     p(r)\!&=\!2
  \end{aligned}$};

\node[above=-1.25mm] at (2) {\textcolor{chroma4gray}{$I(u)\!=\!\{1,3\}\hspace{7mm}f(u)\!=\!3\negphantom{f(u)\!=\!3}\phantom{I(u)\!=\!\{1,3\}}$}};

\node[below=1mm] at (2) {$\begin{aligned}
     s(u)\!&=\!1 \\[-1mm]
     t(u)\!&=\!2 \\[-1mm]
     p(u)\!&=\!4
  \end{aligned}$};

\node[above=-1.25mm] at (3) {\textcolor{chroma4gray}{$I(a)\!=\!\{2,4\}\hspace{7mm}f(a)\!=\!2\negphantom{f(a)\!=\!2}\phantom{I(a)\!=\!\{2,4\}}$}};

\node[below=1mm] at (3) {$\begin{aligned}
     s(a)\!&=\!1 \\[-1mm]
     t(a)\!&=\!1 \\[-1mm]
     p(a)\!&=\!3
  \end{aligned}$};

\node[below=1mm] at (4) {$\begin{aligned}
     \textcolor{chroma4gray}{I(v)}\!&=\!\textcolor{chroma4gray}{\{1\}}\\[-1mm]
     \textcolor{chroma4gray}{f(v)}\!&=\!\textcolor{chroma4gray}{0} \\[-1mm]
     \negphantom{x_1=}x_1\!=\!s(v)\!&=\!1 \\[-1mm]
     t(v)\!&=\!-3 \\[-1mm]
     \negphantom{y_1=}y_1\!=\!p(v)\!&=\!1
  \end{aligned}$};

\node[below=1mm] at (5) {$\begin{aligned}
     \textcolor{chroma4gray}{I(w)}\!&=\!\textcolor{chroma4gray}{\{3\}} \\[-1mm]
     \textcolor{chroma4gray}{f(w)}\!&=\!\textcolor{chroma4gray}{5} \\[-1mm]
     \negphantom{x_3=}x_3\!=\!s(w)\!&=\!0 \\[-1mm]
     t(w)\!&=\!0 \\[-1mm]
     \negphantom{y_3=}y_3\!=\!p(w)\!&=\!4
  \end{aligned}$};

\node[below=1mm] at (6) {$\begin{aligned}
     \textcolor{chroma4gray}{I(b)}\!&=\!\textcolor{chroma4gray}{\{2\}} \\[-1mm]
     \textcolor{chroma4gray}{f(b)}\!&=\!\textcolor{chroma4gray}{3} \\[-1mm]
     \negphantom{x_2=}x_2\!=\!s(b)\!&=\!-1 \\[-1mm]
     t(b)\!&=\!-1 \\[-1mm]
     \negphantom{y_2=}y_2\!=\!p(b)\!&=\!2
  \end{aligned}$};

\node[below=1mm] at (7) {$\begin{aligned}
     \textcolor{chroma4gray}{I(c)}\!&=\!\textcolor{chroma4gray}{\{4\}} \\[-1mm]
     \textcolor{chroma4gray}{f(c)}\!&=\!\textcolor{chroma4gray}{1} \\[-1mm]
     \negphantom{x_4=}x_4\!=\!s(c)\!&=\!2 \\[-1mm]
     t(c)\!&=\!-2 \\[-1mm]
     \negphantom{y_4=}y_4\!=\!p(c)\!&=\!1
  \end{aligned}$};

\draw[->, >=latex, thick, black!75!white, rounded corners=1.5mm] (1) to (1.25,2.75) to (2);
\draw[->, >=latex, thick, black!75!white, rounded corners=1.5mm] (1) to (6.75,2.75) to (3);
\draw[->, >=latex, thick, black!75!white, rounded corners=1.5mm] (2) to (0,1.25) to (4);
\draw[->, >=latex, thick, black!75!white, rounded corners=1.5mm] (2) to (2.5,1.25) to (5);
\draw[->, >=latex, thick, black!75!white, rounded corners=1.5mm] (3) to (5.5,1.25) to (6);
\draw[->, >=latex, thick, black!75!white, rounded corners=1.5mm] (3) to (8,1.25) to (7);

\end{tikzpicture}
\end{center}
\caption{An ultrametric matrix-vector multiplication~$Ax=y$ performed by Algorithm~\ref{algo:tree_mul}. The annotations in gray belong to the input generated by Algorithm~\ref{algo:tree_gen}, those in black are the values Algorithm~\ref{algo:tree_mul} determines.}\label{fig:mult_example}
\end{figure}

\begin{theorem}
Let~$(V,E)$ be an ultrametric tree with root vertex~$r$ constructed by Algorithm~\ref{algo:tree_gen} for an essentially ultrametric matrix~$A\in\mathbb{R}^{n\times n}$ and let~$x$ be a vector in~$\mathbb{R}^n$. Then Algorithm~\ref{algo:tree_mul} with input~$(V,E)$ and~$x$ terminates with output~$y=Ax$.
\end{theorem}
\begin{proof}
In case $n=1$, both procedures of Algorithm~\ref{algo:tree_mul} only activate their if case. So they terminate after their first iteration and sequentially assign~$s(r)=x_1$ in Line~\ref{line:assign_s},~$t(r)=f(r)s(r)=a_{11}x_1$ in Line~\ref{line:assign_t}, $p(r)=t(r)=a_{11}x_1$ in Line~\ref{line:assign_p}, and eventually $y_1=p(r)=a_{11}x_1$ in Line~\ref{line:assign_y}, which shows the correctness of Algorithm~\ref{algo:tree_mul} for~$n=1$.

By Theorem~\ref{thm:correct_tree_gen}, each vertex in~$V$ can be reached from the root~$r$ by a unique directed path. So each vertex in~$V$ is either a leaf or has outgoing edges to child vertices and each vertex except the root~$r$ has a uniquely determined parent. Both procedures of Algorithm~\ref{algo:tree_mul} initially get the root~$r$ as input. A recursion step receiving a vertex~$u$ calls, in case $u$ is not a leaf, recursion procedures for all children of $u$, due to Lines~\ref{line:partial_product_recursion}, \ref{line:for} and~\ref{line:total_product_recursion}. This causes the algorithm to terminate and also shows that both procedures of Algorithm~\ref{algo:tree_mul} are called for each vertex in~$V$ at some point and thus that the assignments in Lines~\ref{line:assign_s} and~\ref{line:assign_y} eventually are realized for each~$i\in\{1,\ldots,n\}$. So the algorithm's output~$y$ is well-defined. Recall for this conclusion that Theorem~\ref{thm:correct_tree_gen} ensures that there is a leaf~$u\in V$ with~$I(u)=\{i\}$ for each~$i\in\{1,\ldots,n\}$. In view of Lines~\ref{line:assign_s} and~\ref{line:partial_product_recursion}, this also implies that
\begin{equation*}
    s(u)=\sum_{\mathclap{~~~~~v\in V:(u,v)\in E}}s(v) =\ldots = \sum_{\mathclap{~j\in I(u)}}x_j.
\end{equation*}
We use this relationship in the following steps whose purpose is to show that indeed~$y=Ax$ holds for~$n\geq 2$. Let us consider an arbitrary~\mbox{$i\in\{1,\ldots,n\}$} and let~$u_0,\ldots,u_m$ be the vertices on the unique directed path in~$(V,E)$ that leads from the root~$r=u_0$ to the vertex~$u_m$ with~$I(u_m)=\{i\}$. Since we discussed the case~$n=1$, we can assume that~$m\geq 1$ and we know that $f(u_m)=a_{ii}$ by Statement~\ref{thm:correct_tree_gen2} of Theorem~\ref{thm:correct_tree_gen}. Line~\ref{line:assign_y} tells us that~$y_i=p(u_m)$. In view of Lines~\ref{line:assign_p} and~\ref{line:assign_t}, this provides us with
\begin{align*}
    y_i &= p(u_m) = p(u_{m-1}) + t(u_m) = \ldots = \sum_{\mathclap{k=0}}^m t(u_k) \\
        &= \Big[\sum_{\mathclap{k=1}}^m \big(f(u_k)-f(u_{k-1})\big)@s(u_k)\Big] + f(u_0)@@s(u_0)\\
        &= \Big[\sum_{\mathclap{k=1}}^m \big((f(u_k)-f(u_{k-1})\big)\sum_{\mathclap{~~j\in I(u_k)}}x_j\Big] + f(u_0)\sum_{\mathclap{~~j\in I(u_0)}}x_j\\
        &= f(u_m)\sum_{\mathclap{~~~j\in I(u_m)}}x_j + \sum_{\mathclap{k=1}}^{m}f(u_{k-1})\Big[\sum_{\mathclap{~~~~~~j\in I(u_{k-1})}}x_j ~-~ \sum_{\mathclap{~~j\in I(u_{k})}}x_j\Big]\\
        &= f(u_m)@@x_i + \sum_{\mathclap{k=1}}^{m}~\sum_{\mathclap{~~~~~~~~~~~j\in I(u_{k-1})\setminus I(u_k)}}f(u_{k-1})@@x_j\\
        &= a_{ii}@@x_i + \sum_{\mathclap{k=1}}^{m}~\sum_{\mathclap{~~~~~~~~~~~j\in I(u_{k-1})\setminus I(u_k)}}a_{ij}@@x_j ~~~= \sum_{\mathclap{j=1}}^{n}a_{ij}@@x_j.
\end{align*}
which is the relation to be shown. Note for the second to last equality that~\mbox{$i\in I(u_k)$} for all~$k\in\{0,\ldots,m\}$ and therefore Statement~\ref{thm:correct_tree_gen3} of Theorem~\ref{thm:correct_tree_gen} tells us that~$f(u_{k-1})=a_{ij}$ for~$j\in I(u_{k-1})\setminus I(u_k)$.
\end{proof}

\begin{corollary}\label{cor:tree_mult_time}
Algorithm~\ref{algo:tree_mul} requires~$\mathcal{O}(n)$ floating-point operations to multiply an essentially ultrametric matrix~$A\in\mathbb{R}^{n\times n}$ given in its tree representation~$(V,E)$ by a vector $x\in\mathbb{R}^n$. 
\end{corollary}
\begin{proof}
Both procedures of Algorithm~\ref{algo:tree_mul} are called exactly once per vertex in $V$. We know that $|V|=2n-1$ by the proof of Theorem~\ref{thm:correct_tree_gen}. Since each call of one of the procedures involves $\mathcal{O}(1)$ floating-point operations, Algorithm~\ref{algo:tree_mul} requires a total of~$\mathcal{O}(n)$ floating-point operations.
\end{proof}
Note that whereas the running time of Algorithm~\ref{algo:tree_mul} is linear, it requires a matrix that has been encoded as an ultrametric tree. By Corollary~\ref{cor:time_tree_gen}, this can be done in quadratic time using Algorithm~\ref{algo:tree_gen}.

\section{Empirical Insights}\label{sec:empirical}
This section is intended to evaluate the practical performance of the algorithms discussed in the previous sections. We begin by presenting computation times for constructing ultrametric trees by Algorithm~\ref{algo:tree_gen} as well as times that matrix-vector multiplications require when using Algorithm~\ref{algo:tree_mul}. We compare this to the effort involved in standard matrix-vector multiplications. By the term \emph{standard} we refer to a routine that determines the matrix-vector product $y=Ax$ simply by computing $y_i = \sum_{j = 1}^{n} a_{ij} x_j$ for each $i \in \{ 1, \ldots , n \}$. The second half of this section extends this investigation to scenarios in which we want to multiply repeatedly.

Our experiments are conducted on randomly generated matrices for whose generation we rely on the following characterization by Fiedler~\cite{fiedler2002remarks}.
\begin{theorem}\label{thm:fiedler}
Up to a simultaneous permutation of rows and columns, each special ultrametric matrix~$A=[a_{ij}]\in\mathbb{R}^{n\times n}$ with~$n\geq 2$ can be obtained by choosing~$n-1$ numbers~$a_{12},a_{23},\ldots,a_{n-1,n}$ and setting
\begin{align*}
    a_{11}=a_{12},\quad a_{ii}&=\max\{a_{i-1,i},a_{i,i+1}\}\text{ for }i=2,\ldots, n-1,\quad a_{nn}=a_{n-1,n},\\
    a_{ik}&=\min\{a_{i,k-1},a_{i+1,k}\}\text{ for all }i,k\text{ where }1\leq i < k-1 \leq n-1,\\
    a_{ki}&=a_{ik}\text{ for all }i,k\text{ where } i>k.
\end{align*}
\end{theorem}
In our tests, the numbers~$a_{12}, a_{23}, \ldots , a_{n-1, n}$ are taken uniformly at random from $\{ 1, \ldots , n-1 \}$ and all the other matrix entries are determined as described in Theorem~\ref{thm:fiedler}. Having generated such a matrix, we randomly perform a simultaneous permutation of its rows and columns. This is to avoid unintended advantages for our algorithms, which is to be expected when the entries of the input matrices are already presorted. As well, the entries of the vectors to be multiplied are chosen uniformly at random from~$\{ 1, \ldots , n-1 \}$. The source code used to generate the data as well as implementations of Algorithms~\ref{algo:tree_gen} and~\ref{algo:tree_mul} are available under Hofmann and Oertel~\cite{Ultrametric_matrix_tools_2021}.
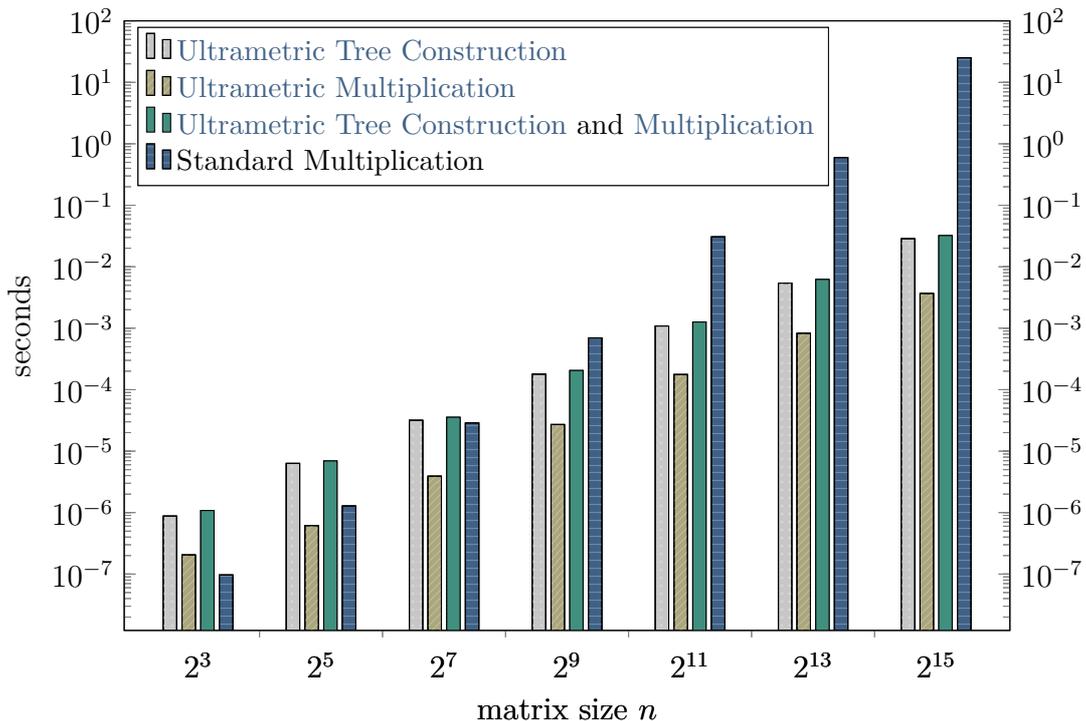
\begin{figure}[ht]
\centering
\pgfplotstableread[col sep=comma]{data_single.csv}\singledata
\begin{tikzpicture}
\begin{semilogyaxis}[y=3.54mm, ybar, log origin=infty, xtick=data, width=0.9\textwidth, height=0.4\textheight, scaled ticks=false, xticklabels={$2^3$, $2^5$, $2^7$, $2^9$, $2^{11}$, $2^{13}$, $2^{15}$}, ylabel=seconds, xlabel=matrix size~$n$, major x tick style = {opacity=0}, minor x tick num = 1, minor tick length=0.5ex, legend style={at={(0.015,0.99)}, anchor=north west}, xtick pos=left, yticklabel pos=left, legend cell align={left}, bar width=5pt, ymax=100]

\addplot[draw=black, fill=chroma4gray, pattern color=chroma4gray!75!white,postaction={pattern=dots}] table[x=pos, y=tree_gen_mean] \singledata;
\addlegendentry{\small\hyperlink{tree_link}{Ultrametric Tree Construction}}

\addplot[draw=black, fill=chroma4sand, pattern color=chroma4sand!75!white, postaction={pattern=north east lines}] table[x=pos, y=tree_mult_mean] \singledata;
\addlegendentry{\small\hyperlink{tree_mult_link}{Ultrametric Multiplication}}

\addplot[draw=black, fill=chroma4green, pattern color=chroma4green!75!white] table[x=pos, y=complete_tree_mult_mean] \singledata;
\addlegendentry{\small\hyperlink{tree_link}{Ultrametric Tree Construction} and \hyperlink{tree_mult_link}{Multiplication}}

\addplot[draw=black, fill=chroma4blue, pattern color=chroma4blue!75!white,postaction={pattern=horizontal lines}] table[x=pos, y=normal_mult_mean] \singledata;
\addlegendentry{\small Standard Multiplication}

\end{semilogyaxis}
\begin{semilogyaxis}[y=3.54mm, ybar, log origin=infty, xtick=data, width=0.9\textwidth, height=0.4\textheight, scaled ticks=false, xticklabels={$2^3$, $2^5$, $2^7$, $2^9$, $2^{11}$, $2^{13}$, $2^{15}$}, xlabel=matrix size~$n$, major x tick style = {opacity=0}, minor x tick num = 1, minor tick length=0.5ex, legend style={at={(0.015,0.99)}, anchor=north west}, xtick pos=left, yticklabel pos=right, legend cell align={left}, bar width=5pt, ymax=100]

\addplot[draw=black, fill=chroma4gray, pattern color=chroma4gray!75!white,postaction={pattern=dots}] table[x=pos, y=tree_gen_mean] \singledata;
\addlegendentry{\small\hyperlink{tree_link}{Ultrametric Tree Construction}}

\addplot[draw=black, fill=chroma4sand, pattern color=chroma4sand!75!white, postaction={pattern=north east lines}] table[x=pos, y=tree_mult_mean] \singledata;
\addlegendentry{\small\hyperlink{tree_mult_link}{Ultrametric Multiplication}}

\addplot[draw=black, fill=chroma4green, pattern color=chroma4green!75!white] table[x=pos, y=complete_tree_mult_mean] \singledata;
\addlegendentry{\small\hyperlink{tree_link}{Ultrametric Tree Construction} and \hyperlink{tree_mult_link}{Multiplication}}

\addplot[draw=black, fill=chroma4blue, pattern color=chroma4blue!75!white,postaction={pattern=horizontal lines}] table[x=pos, y=normal_mult_mean] \singledata;
\addlegendentry{\small Standard Multiplication}

\end{semilogyaxis}
\end{tikzpicture}
\caption{Computation times for matrix-vector products. The results in this chart are averaged over~$10$ runs with random ultrametric matrices as input.}
\end{figure}\label{fig:single}
Figure~\ref{fig:single} shows computation times for a single multiplication of an ultrametric matrix by a vector. We compare the time a standard multiplication takes with the time for multiplying by Algorithm~\ref{algo:tree_mul}. The latter algorithm requires that the input matrix is given in its tree representation. So we additionally consider the time that Algorithm~\ref{algo:tree_gen} needs to construct a corresponding ultrametric tree. All the computation times are averaged over 10 runs with varied matrices and vectors.

For very small matrices, the standard routine is faster than the tree multiplication by Algorithm~\ref{algo:tree_mul}, even without counting the effort for encoding an ultrametric matrix as its associated tree structure. For matrices up to a size of about $n=2^7$, the tree multiplication may be faster than the standard method. However, counting the total duration, including the time required to encode the given matrix as its ultrametric tree, the standard routine is still to be preferred. For larger matrix sizes, the methods we propose consume considerably less time than a standard routine. For example, multiplying a matrix of size $n=2^{15}$ by a vector is about~780 times faster compared to using a standard multiplication. Within our methods, the largest portion of the computation time is required by the tree construction. So applying the proposed methods may especially pay off in situations where we want to multiply repeatedly. To demonstrate this, we conclude this section with an example in which our ultrametric multiplication techniques are used as part of an iterative matrix method.

Suppose we want to compute an approximate solution~$x\in\mathbb{R}^n$ to a system of linear equations~$Ax=b$ with~$b\in\mathbb{R}^n$ where~$A=[a_{ij}]\in\mathbb{R}^{n\times n}$ is a diagonal dominant ultrametric matrix. A classical iterative scheme to solve such a system is the Jacobi method, presented by Golub and Van Loan~\cite[Chapter 11]{golub2013matrix}, for example. The basic idea behind this method is to compute a sequence~$(x^k)$ that, under certain conditions, converges to~$x=A^{-1}b$ by iterating
\begin{equation*}
    x^{k+1} = D^{-1}\left(b-Bx^{k}\right),
\end{equation*}
where~$D\coloneqq\diag(a_{ii} : i=1,\ldots,n)$ contains the diagonal of~$A$ and~$B\coloneqq A-D$ contains the off-diagonal elements of~$A$. Since here the inversion of the diagonal matrix~$D$ is computationally simple, the effort of an iteration is largely determined by the cost of the matrix-vector multiplication~$Bx^k$, for which we propose our ultrametric multiplication techniques. Since a lot of iterative matrix methods rely on repeated matrix-vector multiplications, our techniques may be of use in many of them, provided that the ultrametric structure is preserved throughout the iterations to be performed.

\begin{figure}
\centering
\pgfplotstableread[col sep=comma]{data_jacobi.csv}\jacobidata
\begin{tikzpicture}
\begin{semilogyaxis}[y=3.54mm, ybar, log origin=infty, xtick=data, width=0.9\textwidth, height=0.4\textheight, scaled ticks=false, xticklabels={$2^3$, $2^5$, $2^7$, $2^9$, $2^{11}$, $2^{13}$, $2^{15}$}, ylabel=seconds, xlabel=matrix size~$n$, major x tick style = {opacity=0}, minor x tick num = 1, minor tick length=0.5ex, legend style={at={(0.015,0.99)}, anchor=north west}, xtick pos=left, legend cell align={left}, bar width=5pt, legend style={cells={align=left}}, ymax=5000, ymin=0.000001]

\addplot[draw=black, fill=chroma4gray, pattern color=chroma4gray!75!white,postaction={pattern=dots}] table[x=pos, y=tree_gen_mean] \jacobidata;
\addlegendentry{\small\hyperlink{tree_link}{Ultrametric Tree Construction}}

\addplot[draw=black, fill=chroma4sand, pattern color=chroma4sand!75!white, postaction={pattern=north east lines}] table[x=pos, y=tree_algo_mean] \jacobidata;
\addlegendentry{\small Jacobi Method using \hyperlink{tree_mult_link}{Ultrametric Multiplication}}

\addplot[draw=black, fill=chroma4green, line width=0, pattern color=chroma4green!75!white] table[x=pos, y=complete_tree_algo_mean] \jacobidata;
\addlegendentry{\small\hyperlink{tree_link}{Ultrametric Tree Construction} and Jacobi}
\addlegendimage{empty legend}
\addlegendentry{\small Method using \hyperlink{tree_mult_link}{Ultrametric Multiplication}}

\addplot[draw=black, fill=chroma4blue, pattern color=chroma4blue!75!white,postaction={pattern=horizontal lines}, legend/.style={cells={anchor=west},cells={align=left}}] table[x=pos, y=normal_algo_mean] \jacobidata;
\addlegendentry{\small Jacobi Method using Standard Multiplication}

\end{semilogyaxis}
\begin{semilogyaxis}[y=3.54mm, ybar, log origin=infty, xtick=data, width=0.9\textwidth, height=0.4\textheight, scaled ticks=false, xticklabels={$2^3$, $2^5$, $2^7$, $2^9$, $2^{11}$, $2^{13}$, $2^{15}$}, xlabel=matrix size~$n$, major x tick style = {opacity=0}, minor x tick num = 1, minor tick length=0.5ex, legend style={at={(0.015,0.99)}, anchor=north west}, xtick pos=left, yticklabel pos=right, legend cell align={left}, bar width=5pt, legend style={cells={align=left}}, ymax=5000, ymin=0.000001]

\addplot[draw=black, fill=chroma4gray, pattern color=chroma4gray!75!white,postaction={pattern=dots}] table[x=pos, y=tree_gen_mean] \jacobidata;
\addlegendentry{\small\hyperlink{tree_link}{Ultrametric Tree Construction}}

\addplot[draw=black, fill=chroma4sand, pattern color=chroma4sand!75!white, postaction={pattern=north east lines}] table[x=pos, y=tree_algo_mean] \jacobidata;
\addlegendentry{\small Jacobi Method using \hyperlink{tree_mult_link}{Ultrametric Multiplication}}

\addplot[draw=black, fill=chroma4green, pattern color=chroma4green!75!white] table[x=pos, y=complete_tree_algo_mean] \jacobidata;
\addlegendentry{\small\hyperlink{tree_link}{Ultrametric Tree Construction} and Jacobi}
\addlegendimage{empty legend}
\addlegendentry{\small Method using \hyperlink{tree_mult_link}{Ultrametric Multiplication}}

\addplot[draw=black, fill=chroma4blue, pattern color=chroma4blue!75!white,postaction={pattern=horizontal lines}, legend/.style={cells={anchor=west},cells={align=left}}] table[x=pos, y=normal_algo_mean] \jacobidata;
\addlegendentry{\small Jacobi Method using Standard Multiplication}

\end{semilogyaxis}
\end{tikzpicture}
\caption{Computation times for solving linear systems. The results in this chart are averaged over~$10$ runs with random ultrametric matrices as input.}
\end{figure}
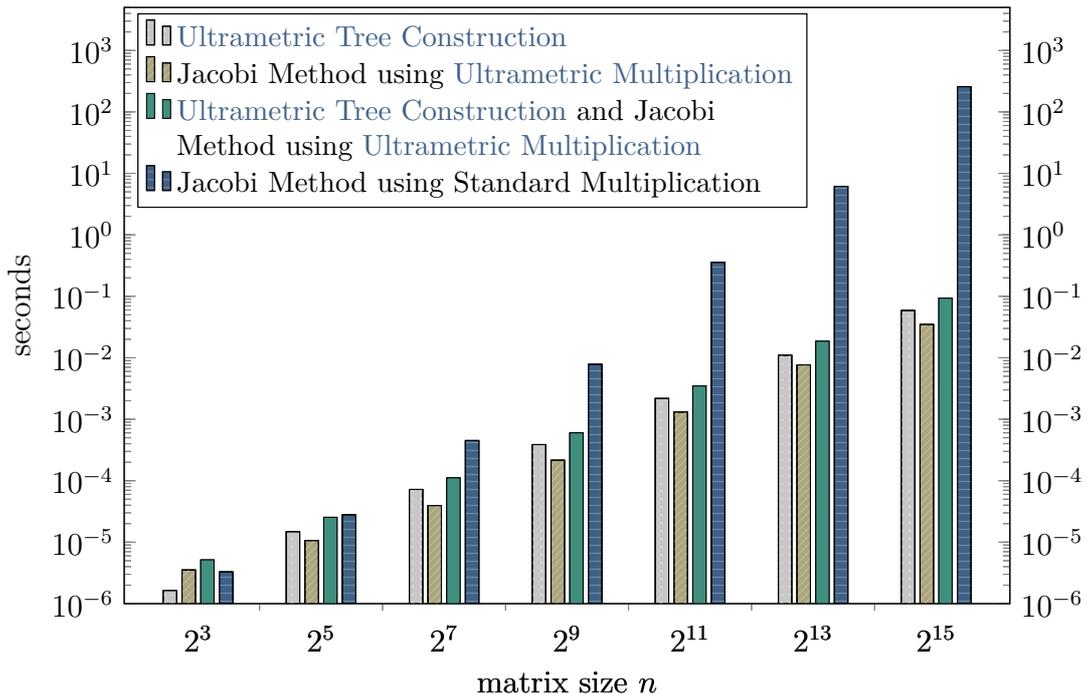\label{fig:multiple}

Figure~\ref{fig:multiple} shows empirical results on the performance of our methods when using them as a subroutine within the Jacobi method to solve a system of linear equations~$Ax=b$. The matrices on which our tests are based are constructed as described above with the only exception that we now require them to be strictly diagonal dominant. More precisely, we choose the diagonal elements~$a_{ii}$ for~$i\in\{1,\ldots,n\}$ uniformly at random from~$\{ d + 1, \ldots , d^2 \}$ where~$d = \sum_{j = 1}^n a_{i j}$, which guarantees a reasonable convergence rate of the Jacobi method.

As with the results described in Figure~\ref{fig:single}, the initial effort involved in the tree construction begins to pay off already for relatively small matrix sizes. For the scenario at hand, the breakpoint is reached at a size of about $n=2^6$, which is earlier than in the experiments illustrated in Figure~\ref{fig:single}. Also, compared to using naive matrix multiplication within an iterative scheme, the difference in performance becomes considerably larger. For example, for a system of size $n=2^{15}$, using our methods within the Jacobi method is about~2750 times faster than the standard version. This underlines the potential of the proposed methods for large scale computations.

\section*{Acknowledgments}
Our research was partially funded by the Deutsche Forschungsgemeinschaft (DFG, German Research Foundation) -- Project-ID 416228727 -- SFB 1410 and by the Wallenberg AI, Autonomous Systems and Software Program (WASP) funded by the Knut and Alice Wallenberg Foundation.

\bibliographystyle{plain}
\bibliography{bibliography}

\begin{thebibliography}{10}

\bibitem{chehreghani2020unsupervised}
Morteza~H. Chehreghani.
\newblock Unsupervised representation learning with minimax distance measures.
\newblock {\em Machine Learning}, 109(11):2063--2097, 2020.

\bibitem{dellacherie2014inverse}
Claude Dellacherie, Servet Mart{\'i}nez, and Jaime~S. Mart{\'i}n.
\newblock {\em Inverse M-Matrices and Ultrametric Matrices}.
\newblock Lecture Notes in Mathematics. Springer, 2014.

\bibitem{diestel2017graph}
Reinhard Diestel.
\newblock {\em Graph Theory}.
\newblock Springer, 2017.

\bibitem{fiedler2000special}
Miroslav Fiedler.
\newblock Special ultrametric matrices and graphs.
\newblock {\em SIAM Journal on Matrix Analysis and Applications},
  22(1):106--113, 2000.

\bibitem{fiedler2002remarks}
Miroslav Fiedler.
\newblock Remarks on monge matrices.
\newblock {\em Mathematica Bohemica}, 127(1):27--32, 2002.

\bibitem{golub2013matrix}
Gene~H. Golub and Charles~F. Van~Loan.
\newblock {\em Matrix Computations}.
\newblock Johns Hopkins University Press, 2013.

\bibitem{gomoryhu1961multi}
Ralph~E. Gomory and Tien~Chung Hu.
\newblock Multi-terminal network flows.
\newblock {\em SIAM Journal}, 9(4):551--570, 1961.

\bibitem{Ultrametric_matrix_tools_2021}
Tobias Hofmann and Andy Oertel.
\newblock {Ultrametric matrix tools}, 2021.
\newblock Version: 0.1.1. \textsc{url:}
  \url{https://doi.org/10.5281/zenodo.5809300}.

\bibitem{hofmann2021edge}
Tobias Hofmann and Uwe Schwerdtfeger.
\newblock Edge-connectivity matrices and their spectra.
\newblock {\em arXiv:2102.04541}, 2021.

\bibitem{holly2001pictures}
Jan~E. Holly.
\newblock Pictures of ultrametric spaces, the $p$-adic numbers, and valued
  fields.
\newblock {\em The American Mathematical Monthly}, 108(8):721--728, 2001.

\bibitem{lauritzen2019maximum}
Steffen Lauritzen, Caroline Uhler, and Piotr Zwiernik.
\newblock Maximum likelihood estimation in gaussian models under total
  positivity.
\newblock {\em The Annals of Statistics}, 47(4):1835--1863, 2019.

\bibitem{little2020path}
Anna~V. Little, Mauro Maggioni, and James~M. Murphy.
\newblock Path-based spectral clustering: guarantees, robustness to outliers,
  and fast algorithms.
\newblock {\em Journal of Machine Learning Research}, 21, 2020.

\bibitem{martinez1994inverse}
Servet Mart{\'i}nez, G{\'e}rard Michon, and Jaime~S. Mart{\'i}n.
\newblock {Inverse of strictly ultrametric matrices are of Stieltjes type}.
\newblock {\em SIAM Journal on Matrix Analysis and Applications},
  15(1):98--106, 1994.

\bibitem{nabben1994linear}
Reinhard Nabben and Richard~S. Varga.
\newblock {A linear algebra proof that the inverse of a strictly ultrametric
  matrix is a strictly diagonally dominant Stieltjes matrix}.
\newblock {\em SIAM Journal on Matrix Analysis and Applications},
  15(1):107--113, 1994.

\end{thebibliography}

\end{document}